\documentclass{amsart}

\newtheorem{theorem}{Theorem}[section]

\theoremstyle{definition}
\newtheorem{definition}[theorem]{Definition}
\newtheorem{example}[theorem]{Example}

\newtheorem{proposition}[theorem]{Proposition}

\newtheorem{corollary}[theorem]{Corollary}

\theoremstyle{remark}
\newtheorem{remark}[theorem]{Remark}

\numberwithin{equation}{section}

\newcommand{\R}{\ensuremath{\mathbb{R}}}
\newcommand{\N}{\ensuremath{\mathbb{N}}}

\begin{document}

\title{Dissipativity and positive off-diagonal property of operators on ordered Banach spaces}

\author{Feng Zhang}
\address{College of Mathematics, Taiyuan University of Technology, Taiyuan 030024, Shanxi, China}
\email{zhangfeng.0631@163.com}

\author{Onno van Gaans}
\address{Mathematical Institute, Leiden University, 2333 CA Leiden, The Netherlands}
\email{vangaans@math.leidenuniv.nl}


\date{June 10, 2020}



\begin{abstract}
In this paper, we provide a sublinear function $p$ on ordered Banach spaces,  which depends on the order structure of the space. With respect to this $p$, we study the relation between  $p$-contractivity of positive semigroups and the $p$-dissipativity of its generators. The positive off-diagonal property of generators is also studied in  ordered vector spaces.
\end{abstract}

\maketitle

\section{Introduction}

It is known that on a Banach space $X$, the positivity and the contractivity 
of a semigroup $(T(t))_{t\ge0}$  can be characterized by means of dissipativity of its generator $A$ with respect to some appropriate sublinear functions, see \cite[Theorem 2.6]{Are1986A}. 
This property is also proved in the case of $X$ being an ordered Banach space with normal positive cone and with respect to a canonical half-norm on $X$,
see \cite[Proposition 7.13]{CleHeiAngDuiPag1987}. 
In Banach lattices, moreover, the positivity of $(T(t))_{t\ge0}$ can be  characterized by its  generator  $A$ which satisfies the ``positive off-diagonal'' property (it is called ``positive minimum principle'' in \cite[Theorem 1.11]{Are1986C}). This is also studied in an ordered Banach space with a positive cone with nonempty interior, \cite[Theorem 7.27]{CleHeiAngDuiPag1987}.
 For more positive off-diagonal properties of operators on ordered normed spaces, we refer the reader to  \cite{arendt1982generalization, kalauch2003positive, Kal2006, KS2013}.

The goal of this paper is to investigate the positivity and contractivity of semigroups through the dissipativity of generators with respect to corresponding sublinear functions on ordered vector spaces. So the choice of sublinear functions on ordered vector spaces is crucial. In Section \ref{Half-norms}, we give two different ways of defining the sublinear functions on ordered vector spaces. One is given by a regular norm and the other one is obtained through a dual function. The latter one turns out to be more effective in studying the positivity and contractivity of semigroups on ordered Banach spaces in Section \ref{Cp-oBs}, these are also our main results. Moreover, in Section \ref{POD-ovs} we will show 
a representation for positive functions on Archimedean pre-Riesz spaces, which will be used to study the positive off-diagonal property of operators on the pre-Riesz space $C^1[0,1]$.

\section{Preliminaries}

In this section, we will collect some basic terminology, which we mainly refer to \cite{AliBur1985, Are1986A, Kalauch-vanGaans}.
Let $X$ be a real vector space, and $K\subseteq X$. $K$ is said to be a \emph{cone} in $X$ if $x, y\in K$, $\lambda, \mu \in \R^+$ implies $\lambda x+\mu y\in K$ and $K\cap (-K) =\{0\}$. 
A partial order in $X$ is induced by $x\le y$ whenever $y-x\in K$, we then say $(X,K)$ is a \emph{(partially) ordered vector space}. 
$(X, K, \|\cdot\|)$ is called an \emph{ordered Banach space} whenever it is norm complete.
The space of all bounded linear operators on $X$ is denoted by $L(X)$, the positive operators in $L(X)$ are $L(X)^+:=\{T\in L(X): Tx\ge 0, \ \forall 0\le x\in X\}$. 
Then $(L(X), L(X)^+)$ is an ordered vector space.
The domain of an operator $T$ on $X$ is denoted by $\mathcal{D}(T)$. 
A one-parameter semigroup of operators is  written usually as $(T(t))_{t\geq 0}$, and it is  called \emph{strongly continuous} if the map $t\mapsto T(t)$ is continuous for the strong topology on $L(X)$. 

Let $(X, K, \|\cdot\|)$ be an ordered normed vector space. A function $p$ on $X$ is called \emph{sublinear} if $p(x+y)\le p(x)+p(y)$ for all $x, y\in X$ and $p(\lambda x)=\lambda p(x)$ for all $x\in X$, $\lambda\ge 0$.
It is clear that $p(0)=0$. 
If a sublinear function $p$ on $X$ satisfies  $p(x)+p(-x)>0$ whenever $x\neq0$ in $X$, then $p$ is called  a \emph{half-norm}. If  there exists a constant $M>0$ such that $p(x)+p(-x)\geq M\|x\|$ for all $x\in X$, then $p$ is  called a \emph{strict half-norm}. 
Let $x\in X$, define
\begin{align}\label{canonical half norm}
\psi(x)=\mbox{dist}(-x, K)&=\inf \{\|x+y\|\colon y\in K\}\nonumber \\ 
&=\inf\{\|y\|\colon x\le y,  y\in X\},
\end{align} 
then $\psi$ is a half-norm in $X$. This half-norm $\psi$ is called the  \emph{canonical half-norm}. 
A seminorm $p$ on $X$ is called \emph{regular} if for all $x\in X$ one has
\[p(x):=\inf\{p(y)\colon y\in X \mbox{ such that } -y\leq x\leq y\}.\] 
A regular seminorm on an ordered vector space will be denoted by   $\|\cdot\|_r$. 
If $(X,K)$ is an ordered  vector space equipped with a seminorm $p$, 
then  for $x\in X$,  $\|x\|_r:=\inf\{p(y)\colon y\in X, -y\le x \le y \}$ is called the \emph{regularization} of the original seminorm $p$. 
Moreover,  $\|\cdot\|_r$ is a seminorm,  and  $\|\cdot\|_r\le p(\cdot)$ on $K$.
A seminorm $p$ on $X$ is called \emph{monotone} if for every $x, y \in X$ such that $0\le x\le y$ one has that $p(x)\le p(y)$. 

If $(M, \|\cdot\|)$ is any normed vector space, we denote by $M'$ the continuous linear functionals on $M$.

Let $(X, K, \|\cdot\|)$ be an ordered normed vector space, $p$ be a  sublinear function on $X$.
A bounded operator $T$ on $X$ is called \emph{$p$-contractive} if $p(Tx)\leq p(x)$ for all $x\in X$. 
Similarly, a semigroup $(T(t))_{t\geq0}$ is called \emph{$p$-contractive} if $T(t)$ is $p$-contractive for all $t\geq0$.
The  \emph{subdifferential} of  $p$ at point $x\in X$, denoted by $dp(x)$, is defined  as
 \begin{equation}
\label{eq:cone2}
dp(x)=\{x'\in K' \colon  \langle y, x'\rangle\le p(y) \mbox{ for all } y\in X,
\langle x, x'\rangle =p(x)\}.
\end{equation}
\begin{definition}
An operator $A\colon \mathcal{D}(A) \subseteq X\to X$ is called \emph{$p$-dissipative} if for all $x \in \mathcal{D}(A)$ there exists $y'\in dp(x)$ such that $\langle Ax, y'\rangle\le0$; $A$ is called $strictly$ $p$-dissipative if for all $x\in \mathcal{D}(A)$ the inequality $\langle Ax,y'\rangle\le0$ holds for all $y'\in dp(x)$.
\end{definition}

Let $(X, \|\cdot\|)$ be a Banach lattice,  $p$ be a canonical half-norm on  $X$. 
We have that $p(x)=\|x^+\|$ for all $x\in X$. 
In fact, let $x\in X$, $p(x)\le \|x^+\|$  is obvious. Since $|x+y|\ge (x+y)^+\ge x^+$ for any $0\le y\in X$, we have that $\|x+y\|\ge \|x^+\|$. So $p(x)\ge \|x^+\|$. 
To make a difference from  the above, we will use  distinctive notations.
 Let $N^+\colon X\rightarrow \R$ be the function given by  $N^+(x)=\|x^+\|$. 
 Then the subdifferential of $N^+$  at point $x \in X$ is 
\[dN^+(x)=\{x'\in K'\colon \|x'\|\le 1, \langle x, x' \rangle=\|x^+\|\}.\]
An operator $T$ on $X$ is called \emph{(strictly) dispersive} if $T$ is (strictly) $N^+$-dissipative. 

We continue by some notations of pre-Riesz theory. 
An ordered vector space $(X,K)$   is called \emph{directed} if for every $x, y\in X$ there exists $z\in X$ such that $z\ge x, z\ge y$, and is called  \emph{Archimedean} if for every $x, y\in X$ with $nx\le y$ for all $n\in \N$ one has $x\le 0$. 
We say that $X$ is  a \emph{pre-Riesz space} if for every $x, y, z \in X$ the inclusion $\{x+y, x+z\}^{\text{u}} \subseteq\{y, z\}^{\text{u}}$ implies $x \in K$, where $\{x, y\}^{\text{u}}=\{z\in X\colon z\ge x, z\ge y\}$. 
Every   directed Archimedean  ordered vector space is pre-Riesz \cite{Kalauch-vanGaans}.  A  linear subspace $D$ of $X$ is called \emph{order dense} in $X$  if 
\[x=\inf\{d\in D\colon d\ge x\}, \forall x\in X.\]
Every pre-Riesz space admits a Riesz completion, and it is an order dense subspace of the Riesz completion \cite{Kalauch-vanGaans}.  A linear map $i\colon X\rightarrow Y$ between two ordered vector spaces is called \emph{bipositive} if for every $x\in  X$ one has $x\ge 0$ if and only if $i(x)\ge 0$.

\section{Half-norms}\label{Half-norms}

In this section, we will introduce two different seminorms on  ordered vector spaces. One is induced by the regular norms, and the other one involves the order structure which will be used in the following section.

\begin{proposition}
 Let $(X, K, \|\cdot\|)$ be an ordered normed vector space, and $\|\cdot\|_r$ the regularization of $\|\cdot\|$,   and let $p$ be defined as 
\begin{eqnarray}\label{itemi}
p(x)=\inf\{\|y\|_r\colon \, y\in X, y\geq0, y\geq x\}, \forall x\in X.
\end{eqnarray}
Then  $p$  is a strict half-norm on $X$. Moreover, if $\psi$ on $X$ be defined by (\ref{canonical half norm}), then $p(x)=\psi(x)$ for all $x\in K\cup (-K)$, in particular, $p(x)=\psi(x)=0$ for $x\in (-K)$.
\begin{proof}
 Let $x\neq 0$ be in $X$. Since $\|\cdot\|_r$ is the regularization of $\|\cdot\|$, we have that $\|x\|_r:=\inf\{\|y\| \colon y\in X, -y\le x \le y \}$ with $y\neq 0$. As $\|\cdot\|$ is a norm, $\|y\|\neq 0$. 
 So  we have that $\|x\|_r>0$. 
   If $y\ge 0$, $y\ge x$ and $z\ge0$, $z\ge-x$, then $-(y+z)\le x \le (y+z)$. 
Since $\|\cdot\|_r$ is regular, we have $\| x\|_r \le \|y+z\|_r $.  So $\|y\|_r+\|z\|_r\ge \|y+z\|_r\ge \|x\|_r>0$. Hence, by passing to the infimum in each term we get $p(x)+p(-x)\ge \|x\|_r>0$. Thus $p$ is a strict half-norm. 
 
Moreover, if $0\le x \in X$, we can take $y=x\ge 0$ in (\ref{itemi}), then $p(x)=\|x\|_r=\psi(x)$.  It is clear that  $p(x)=\psi(x)=0$ for $x\le 0$ in $X$.
\end{proof}
\end{proposition}

\begin{proposition}\label{regular-contractive}

 Let $(X, K, \|\cdot\|)$ be an ordered normed vector space, and $\|\cdot\|_r$ the regularization of $\|\cdot\|$,   and let $p$ be defined by \eqref{itemi}. 
 If  $T\in L(X)^+$ is a contractive operator with respect to the regular norm $\|\cdot\|_r$, 
 then $T$ is $p$-contractive.

\begin{proof}
Let $T\in L(X)^+$. By the assumption, we have that $\|Ty\|_r\le \|y\|_r$ for all $0\le y\in X$. Then for $x\in X$,
$p(x)=\inf\{\|y\|_r\colon \, y\ge 0, y\ge x\} \ge\inf\{\|Ty\|_r\colon \, y\ge 0, y\ge x\}$. 
Since $T$ is a positive operator, we have that $y\ge 0, y\ge x$ implies $Ty\ge 0, Ty\ge Tx$. Hence, $\inf\{\|Ty\|_r\colon \, y\ge 0, y\ge x\}\ge \inf\{\|Ty\|_r\colon \, Ty\ge 0, Ty\ge Tx\} =p(Tx)$.
So $T$ is $p$-contractive.
\end{proof}

\end{proposition}

We continue with a different approach of sublinear functions on ordered vector spaces, which turns out to be more useful in dealing with the contractivity and positivity of semigroups.

Let $\phi \in X'$,  define $p_\phi$ on $X$ by
\begin{eqnarray}\label{itemii}
p_{\phi}(x)=\inf\{\langle y, \phi\rangle\colon \, y\in X,  y\geq0, y\geq x\}, \, \forall\, x\in X.
\end{eqnarray}

\begin{proposition}
Let $X$ be an ordered vector space, and let $\phi\in X'$. If 
$p_{\phi}$ is defined by \eqref{itemii}, then $p_{\phi}$ is sublinear. Moreover, $p_{\phi}(x)=\langle x, \phi\rangle$ if $x\ge 0$, and $p_{\phi}(x)=0$ if $x\le 0$. 
\end{proposition}
\begin{proof}
Let $\phi \in X'$, and  $x, y\in X$. By the definition, we have that $p_{\phi}(x+y)
=\inf\{\langle z, \phi\rangle\colon \, z\in X,  z\geq0, z\geq x+y\}$. Let $ z_1+z_2=z$, we have that 
$\inf\{\langle z, \phi\rangle\colon \, z\in X,  z\geq0, z\geq x+y\}=\inf\{\langle z_1+z_2, \phi\rangle\colon \, z_1+z_2\in X,  z_1+z_2\geq0, z_1+z_2\geq x+y\}$.
Since $z_1\in X, z_2\in X,  z_1\geq0,  z_2\geq0, z_1\geq x, z_2\geq y$ implies $z_1+z_2\in X,  z_1+z_2\geq0, z_1+z_2\geq x+y$, we have that $\inf\{\langle z_1+z_2, \phi\rangle\colon \, z_1+z_2\in X,  z_1+z_2\geq0, z_1+z_2\geq x+y\}\le\inf\{\langle z_1+z_2, \phi\rangle\colon \, z_1\in X, z_2\in X,  z_1\geq0,  z_2\geq0, z_1\geq x, z_2\geq y\}$. It is clear that $\inf\{\langle z_1+z_2, \phi\rangle\colon \, z_1\in X, z_2\in X,  z_1\geq0,  z_2\geq0, z_1\geq x, z_2\geq y\}
=\inf\{\langle z_1, \phi\rangle\colon \, z_1\in X,  z_1\geq0,   z_1\geq x\}+\inf\{\langle z_2, \phi\rangle\colon \, z_2\in X,  z_2\geq0,   z_2\geq y\}=p_{\phi}(x)+p_{\phi}(y)$. So we have the subadditivity.
The positive homogeneity of $p_{\phi}$ is obvious. So $p_{\phi}$ is sublinear.

It is clear that 
$p_{\phi}(x)=\langle x, \phi\rangle$ for $x\ge 0$, and $p_{\phi}(x)=0$ for $x\le 0$. 
\end{proof}

\begin{proposition}
If $X$ is a vector lattice, $\phi\in X'$ and  $p_{\phi}$ is defined by \eqref{itemii}. Then $p_{\phi}(x^+)=p_{\phi}(x)$ for  $x\in X$. 
\end{proposition}
\begin{proof}
Let $X$ be a vector lattice, $\phi\in X'$ and $x\in X$. By \eqref{itemii}, we have
\begin{align*}
p_{\phi}(x^+)&=\inf\{\langle y,\phi\rangle\colon \, y\ge x^+, y\ge 0\}\\
&=\inf\{\langle y,\phi\rangle\colon \, y\ge \inf\{z\colon z\ge x, z\ge 0\}, y\ge 0\}\\
&\ge\inf\{\langle y,\phi\rangle\colon \, y\ge x, y\ge 0\}=p_{\phi}(x).
\end{align*}
Because  $y\ge x$ and $y\ge 0$ implies $y\ge x^+$, and $\phi$ is monotone, 
we have that $\{\langle y,\phi\rangle\colon \, y\ge x, y\ge 0\}\subseteq \{\langle y,\phi\rangle\colon \, y\ge x^+, y\ge 0\}$, and hence $\inf\{\langle y,\phi\rangle\colon \, y\ge x, y\ge 0\}\ge \inf\{\langle y,\phi\rangle\colon \, y\ge x^+, y\ge 0\}$. So we have $p_{\phi}(x)\ge p_{\phi}(x^+) $. 
\end{proof}

We note that the continuous differential function space $X=C^1[0,1]$  and the Sobolev space $X=W^{n,p}$ are  (pointwise) ordered vector spaces but not lattices. They are norm complete with respect to $\|\cdot\|_{\infty}$ and $\|\cdot\|_{W^{n,p}}$ respectively. However, these norms are not monotone, but one could still define a sublinear function $p_{\phi}$ by \eqref{itemii}. Next, we will study the contractivity of $(T(t))_{t\ge0}$ with respect to $p_{\phi}$ given by \eqref{itemii} on ordered Banach space in the following section.

\section{Contractivity and positivity of semigroups on ordered Banach spaces}\label{Cp-oBs}

In this section, we will study the relation between the $p_\phi$-contractivity of  $(T(t))_{t\ge0}$ and the strictly $p_\phi$-dissipativity of its generator $A$   for  
 $X$ being an ordered Banach space, the sublinear function $p_\phi$ on $X$ is defined by (\ref{itemii}). Firstly,  we will  give a sufficient condition  under which $(T(t))_{t\ge0}$ is contractive with respect to $p_\phi$ on an ordered Banach space.

We will use $X$ to denote the ordered Banach space in the following of this section, and $p_\phi$ on $X$ is defined by (\ref{itemii}) with respect to $\phi\in X'$.

\begin{theorem}\label{contra}
 Let $A\colon  X\supseteq \mathcal{D}(A)\to X$  be a $p_\phi$-dissipative operator. If $(I-\lambda A)$ is invertible for some $\lambda> 0$, then $(I-\lambda A)^{-1}$ is $p_\phi$-contractive. 
 \end{theorem}
\begin{proof}
For a fixed $x\in \mathcal{D}(A)$, let $\psi\in dp_{\phi}(x)$ be such that $\langle Ax, \psi\rangle\le 0$.
Then $\langle y,\psi\rangle\le p_\phi(y)$ for all $y\in X$, and $\langle x,\psi\rangle= p_\phi(x)$. Since $\psi$ is positive on $X$, we have that  $|\langle y, \psi\rangle|\le p_{\phi}(y)$.  For some $\lambda >0$ one has that
\begin{align*}
& p_\phi((\lambda I-A)x)
\ge |\langle (\lambda I-A)x, \psi\rangle| 
\ge \mbox{Re}\langle(\lambda I-A)x, \psi\rangle\\
=&\mbox{Re}(\langle\lambda x, \psi\rangle-\langle Ax, \psi\rangle)\ge \mbox{Re}\langle\lambda x, \psi\rangle
= \mbox{Re}\lambda\langle x, \psi\rangle 
=\mbox{Re}\lambda p_\phi(x)
=\lambda p_\phi(x).
\end{align*}
So if $\lambda>0$ is such that $(I-\lambda A)$ is invertible, then
$(I-\lambda A)^{-1}$ is $p_\phi$-contractive for some $\lambda> 0$.
\end{proof}

Observe that $(I-\lambda A)^{-1}=\frac{1}{\lambda}\left(\frac{1}{\lambda} I-A\right)^{-1}=\frac{1}{\lambda}R\left(\frac{1}{\lambda}, A\right)$. If  a strongly continuous semigroup $(T(t))_{t\ge 0}$ is  generated by the operator $A$, then  for $x$ in $X$, one has
\[
 T(t)x
 =\lim_{n\rightarrow\infty}\left[\frac{n}{t}R\left(\frac{n}{t}, A\right)\right]^n x= \lim_{n\rightarrow \infty}\left(I-\frac{t}{n}A\right)^{-n}x.\]
Hence, by Theorem \ref{contra}, if $A$ is $p_\phi$-dissipative, then  $T(t)$ is $p_\phi$-contractive for every $t\ge 0$.
As a consequence we have the following corollary. 
\begin{corollary}\label{contrT}
If  a strongly continuous semigroup $(T(t))_{t\ge 0}$ is  generated by   a $p_\phi$-dissipative operator $A$, and  $(I-\lambda A)$ is invertible for some $\lambda> 0$.
 Then $T(t)$ is $p_\phi$-contractive for every $t\ge 0$.
\end{corollary}

Next, we will study the positivity of strongly continuous semigroup $(T(t))_{t\ge0}$ on an ordered Banach space. The generator $A$ of $(T(t))_{t\ge0}$ is required to be $p_\phi$-dissipative for all $\phi$  in a total subset of an ordered Banach space.


\begin{definition}
A nonempty subset $\Phi\subseteq K'$ is called \emph{total} in $X$ if $\phi(x)\ge0$ for each $\phi\in \Phi$ implies $x\ge0$.
\end{definition}

The following example shows that the intersection of a subdifferential set and a total subset on an ordered Banach space $X=C^1[0,1]$ is nonempty. 

\begin{example}
Consider $X=C^1[0,1]$, and $\Phi=\{\delta_t\colon \, t\in [0,1]\}$, where $\delta_t(x)=x(t)$ for an arbitrary $x\in X$. Since $\delta_t(x)=x(t)\ge 0$ for all $t\in [0,1]$ implies $x\ge 0$ in $X$. So  $\Phi$ is total. 
We claim that for a given $t\in [0,1]$ and $x\in X$, $dp_t(x)\cap \Phi\neq \emptyset$ for $p_t(x)=\inf\{\delta_t(y)\colon \, y\ge 0, y\ge x\}$. In fact,  let $t\in [0,1]$ and $x\in X$, then $p_t(x)=\inf\{\delta_t(y)\colon \, y\ge 0, y\ge x\}=\inf\{y(t)\colon \, y\ge 0, y\ge x\}=x^+(t)$. (Here, as $X=C^1[0,1]$ is a partially ordered vector space, the positive part is defined in such a way, $x^+=\inf\{y\in X\colon y\ge x, y\ge 0\}$, $\forall x\in X$).  So if $x(t)\le 0$ then $p_t(x)=x^+(t)=0$, and $dp_t(x)\cap \Phi=0$.  If  $x\ge 0$, then $\delta_t(y)=y(t)\le y^+(t)=p_t(y)$, and $\psi(x)=\delta_t(x)=x(t)=x^+(t)=p_t(x)$. So $\delta_t\in dp(x)$. Thus we have shown that $ \Phi \subset dp_t(x)$.

\end{example}


\begin{theorem}\label{posi-semi}
Let $A\colon  X\supseteq \mathcal{D}(A)\to X$ be the generator of strongly continuous semigroup $(T(t))_{t\ge 0}$. If $A$ is $p_\phi$-dissipative for all $\phi$ in a total set $\Phi$ and $(I-\lambda A)$ is invertible for some $\lambda \ge 0$, then $T(t)$ is positive for all $t\ge 0$.
\end{theorem}
\begin{proof}
 Let $\phi\in \Phi$, select  $x\le 0$ in $X$. By Theorem \ref{contra}, for the semigroup $(T(t))_{t\ge 0}$ is generated by the operator $A$  
 one has   $p_\phi(T(t)x)\le p_\phi(x)$ for all $t\ge 0$, which means
\[\inf\{\langle y, \phi\rangle\colon \, y\ge T(t)x, y\ge 0\}\le \inf\{\langle z, \phi\rangle\colon \, z\ge x, z\ge 0\}.\]
Take $z=0$, then the right side of the above inequality is 0. It follows that 
\[\inf\{\langle y, \phi\rangle\colon \, y\ge T(t)x, y\ge 0\}\le 0.\]
Because of $y\ge T(t)x$, then $\langle y, \phi\rangle\ge \langle T(t)x, \phi\rangle$ for $\phi$ is positive. So 
\[\langle T(t)x, \phi\rangle\le \inf\{\langle y, \phi\rangle\colon \, y\ge T(t)x, y\ge 0\}\le 0.\]
Since $\Phi$ is total,  one has that $\langle T(t)x, \phi\rangle\le 0$ for all $\phi\in \Phi$ implies $T(t)x\le0$. Thus $T(t)$ is positive. 
\end{proof}

 \begin{remark}
Due to  \cite[Theorem 1.2]{Are1986C}, if $A$ is densely defined on a Banach lattice, then $(T(t))_{t\ge 0}$ is positive and contractive if and only if $A$ is dispersive and $(\lambda I-A)$ is surjective for some $\lambda>0$. 
By Theorem \ref{contra}, we could generalize one direction of this conclusion to ordered Banach spaces. We will illustrate this through an example  of a second derivative operator with Dirichlet boundary condition. This example originally comes from \cite[Example 1.5]{Are1986C}.
\end{remark}

\begin{example}
Let $X=(C^1[0,1], \|\cdot\|_{\infty})$ be an ordered Banach space, the densely defined operator $A$ be the second derivative operator with Dirichlet boundary condition. Then the domain satisfies 
$\mathcal{D}(A)=\{x\in C^{3}[0,1]\colon\, x(0)=x(1)=x''(0)=x''(1)\}$. 
For $x\in X$, we choose $t\in [0,1]$ such that $p_{\delta_t}(x)=\langle x, \delta_t\rangle=x(t)=\sup_{s\in [0,1]}x(s)=\|x\|_\infty$. Then  $\langle y, \delta_t\rangle=y(t)\le \|y\|_\infty=p_{\delta_t}(y)$ for all $y\in X$. Hence $\delta_t\in dp_{\delta_t}(x)$. Moreover, since $x$ has a maximum in $X$, we have that $\langle Ax, \delta_t\rangle=x''(t)\le 0$. 
So $A$ is $p_{\delta_t}$-dissipative. For $y\in X$ define the function $x_0(t)=\frac{1}{2}[e^t\int_t^1e^{-s}g(s)ds-e^{-t}\int_t^1e^sg(s)ds]$. 
Then there exist $m,n\in \mathbb{R}$ such that $x(t)=x_0(t)+me^t+ne^{-t}$ and $x(0)=x(1)=0$, and then $x\in \mathcal{D}(A)$. Since $x-x''=x_0-x_0''=y$, we have $(I-A)$ is surjective.
For $x\in \mathcal{D}(A)$, suppose that $(I-A)x=0$, then $x(t)=\alpha e^t+\beta e^{-t}$. Notice that $x(0)=x(1)=0$ such that $\alpha=\beta=0$, so $x(t)=0$ for $t\in [0,1]$. So $(I-A)$ is injective. 
It follows from Corollary \ref{contrT} that $A$ is the generator of a  contractive semigroup. By Theorem \ref{posi-semi}, $A$ generates a strongly continuous positive semigroup $T(t)_{t\ge 0}$.


\end{example}

\begin{remark}

Note that in an ordered Banach space, specifically $C^1[0,1]$, it is hard to give the definition of the dispersivity of an operator $A$ with  respect to the original norm, because $C^1[0,1]$ is not a lattice.
 However, by the above discussion, we still have flexibility to choose a function  as in (\ref{itemii}),  which  depends on a function $\phi$ in $X'$.   This is also different from the arguments in \cite[Example 1.5]{Are1986C}.

\end{remark}

\section{Positive off-diagonal property of operators on ordered vector spaces}\label{POD-ovs}

In this section, we will introduce the positive off-diagonal property especially on pre-Riesz spaces, in particular $C^1(\Omega)$ with $\Omega\subseteq \R^n$ open.
Explicitly, we investigate a representation theorem for positive linear functionals on Archimedean pre-Riesz spaces, which is also interesting independently.

\begin{definition}
Let $X$ be an ordered vector space.
A linear operator $A\colon \mathcal{D}(A)\subseteq X \rightarrow X$ is said to have  the \emph{positive off-diagonal property} if $\langle Ax, \phi\rangle\ge 0$ whenever $0\le x\in \mathcal{D}(A)$ and $0\le \phi\in X'$ with $\langle x, \phi \rangle=0$.
\end{definition}

The motivation of the positive off-diagonal property comes from matrix theory, where the off-diagonal elements of the matrix $A=(a_{ij})$ are positive, i.e., $a_{ij}\ge 0$ for all $i\neq j$. It is shown in  \cite[Lemma 7.23]{CleHeiAngDuiPag1987} that on an ordered Banach space $X$ with an order unit $u$ such that $u\in \mathcal{D}(A)$, the operator $A$ has the positive off-diagonal property and $Au\le 0$ if and only if $A$ is $\Psi_u$-dissipative, where $\Psi_u$ is the order unit function, i.e. $\Psi_u(x)=\inf\{\lambda\ge 0\colon x\le \lambda u\}$, $x\in X$.
 However, in general, the properties that $A$ has the positive off-diagonal  property and $A$ is $p$-dissipative for a given $p$ on $X$ do not imply each other, as the following example shows.

\begin{example}\label{POD-dissipa}
Let $X=\R^2$ and $p(x)=\sqrt{x_1^2+x_2^2}$ for $x=(x_1, x_2)$.
Let
$A =\begin{pmatrix}
1 & 1 \\
1 & 1
\end{pmatrix}$, then $A$ has the positive off-diagonal  property. Take $x=(1,0)$, then $x'=(1,0)\in dp(x)$ but $\langle Ax, x'\rangle=1$, so $A$ is not $p$-dissipative. Let $A =\begin{pmatrix}
-1 & -1 \\
1 & 1
\end{pmatrix}$, and $\mathcal{D}(A)=\{x=(x_1, x_2)\in X\colon x_1\ge 0, x_2=0\}$. Take $x'=(1,0)$, then $x'\in dp(x)$ for every $x\in \mathcal{D}(A)$. It is obvious that $\langle Ax, x'\rangle\le 0$. So  $A$ is $p$-dissipative, but does not have positive off-diagonal  property. 
\end{example}

Next, we consider a representation theorem in pre-Riesz spaces.

\begin{theorem}\label{repre-pre-R}
Let $X$ be an Archimedean pre-Riesz space with order unit.
Then there exists a compact Hausdorff space $\Omega$ and a bipositive linear map $i\colon X\rightarrow C(\Omega)$ such that $i(X)$ is order dense in $C(\Omega)$. Moreover, for every positive linear functional $\phi$ on $X$, there exists a regular Borel measure $\mu$ on $\Omega$ such that 
\[\phi(x)=\int_\Omega i(x)(\omega)d\mu(\omega), \quad  x\in X, \ \omega\in \Omega.\]
\end{theorem}

\begin{proof}
The first part of this theorem follows from \cite[Lemma 6]{KalLemGaa2014}.

For the second part, let $C(\Omega)$ be the space of all continuous functions on $\Omega$,  where $\Omega$ is a compact Housdorff space.
Let $i\colon X\rightarrow C(\Omega)$ be a bipositive linear map such that
 $i(X)$ is an order dense subspace of $C(\Omega)$. 
So for a positive  linear function $\phi\colon X\to \mathbb{R}$,
one has that $\phi\circ i^{-1}\colon i(X)\to \mathbb{R}^+$ is a positive linear function on $i(X)$. 
Since $\mathbb{R}$ is Dedekind complete, and $i(X)$ is a majorizing subspace of $C(\Omega)$, 
by Kantorovich's extension theorem (see \cite[Theorem 1.32]{AliBur1985}), there exists an extension of $\phi\circ i^{-1}$ to a positive function $\psi\colon C(\Omega)\to \mathbb{R}$. 
By the Riesz representation theorem, for $\psi$ on $C(\Omega)$, 
there exists a unique regular Borel measure $\mu$ on $\Omega$ such that 
\[\psi(f)=\int_{\Omega}f(\omega)d\mu(\omega), \quad \forall f\in C(\Omega), \ \omega\in \Omega.\]
So for every $x\in X$, one has $\phi\circ i^{-1}(i(x))=\psi(i(x))$. If we take $f=i(x)$, then
\[\phi(x)=\phi\circ i^{-1}(i(x))=\psi(i(x))=\int_{\Omega}i(x)(\omega)d\mu(\omega).\]
Thus we get the conclusion.
\end{proof}

We give an example to illustrate that the positive off-diagonal property of $A$ can be generalized to a special kind of  ordered vector space, in particular to the pre-Riesz space $C^1[0,1]$. 

\begin{example}
Let $C[0,1]$ be the real continuous functions.
Let $X=C^1[0,1]$ which is an Archimedean pre-Riesz space, then $X$ is  an order dense subspace of 
$C[0,1]$. Let $A\in L(X)$ be a densely defined operator, we claim that $A$ has positive off-diagonal  property if 
and only if $(Ax)(t)\ge 0$ whenever $0\le x\in \mathcal{D}(A)$ and $t\in [0,1]$ with $x(t)=0$.
In fact, first suppose that $A$ has the positive off-diagonal property and $0\le x\in \mathcal{D}(A)$, $t\in [0, 1]$ with $x(t)=0$. 
Take $0\le \delta_t\in X'$ to be the point evaluation at $t$, then it follows from $x(t)=\langle x,\delta_t\rangle=0$
that $\langle Ax, \delta_t\rangle\ge 0$, i.e. $(Ax)(t)\ge 0$.
Conversely, assume $0\le x\in \mathcal{D}(A)$, and $0\le \phi\in X'$ with $\langle x, \phi\rangle=0$. Then Theorem \ref{repre-pre-R} can be applied since $C[0, 1]$ has an order unit, and then there exists a regular Borel measure $\mu$ on $[0, 1]$ that represents $\phi$, i.e. 
$\langle x, \phi\rangle=\int_0^1 i(x)(t)d\mu(t)$.
By assumption, we have $i(x)(t)=0$ and $x(t)=0$ for all $t$ in
 the support of $\mu$, then $(Ax)(t)\ge 0$ and $i(Ax)(t)\ge 0$.
 Hence $\langle Ax,\phi\rangle=\int_0^1 i(Ax)(t)d\mu(t)\ge 0$. 
 This shows that $A$ has the positive off-diagonal  property.
\end{example}

\bibliographystyle{amsplain}

\bibliography{journal}






\end{document}